\documentclass{svproc}

\usepackage{type1cm}        
\usepackage{graphicx}      

\usepackage[bottom]{footmisc}

\usepackage{nicefrac}
\usepackage{mathrsfs}
\usepackage{amsmath,amssymb,amscd} 
\usepackage{mathtools}
\usepackage{bbm}
\usepackage{enumitem}
\usepackage{caption} 
\usepackage{tikz}
\usetikzlibrary{calc,decorations.markings,arrows.meta,positioning}

\newcommand{\initialize}{
    \coordinate (ll) at (-1.5,-.75);
    \coordinate (ur) at (1.5,.75);
    \coordinate (basec) at (0,-.75);
    \coordinate (Adelim) at ($(ll)!.8!(ur)$);
    \coordinate (All) at (ll -| Adelim);
    \coordinate (Aul) at (Adelim |- ur);
    \coordinate (Ac) at ($(All)!.5!(ur)$);
    \coordinate (ATdelim) at ($(ll)!.5!(ur)$);
    \coordinate (ATlr) at (ll -| ATdelim);
    \coordinate (ATur) at (ATdelim |- ur);
    \coordinate (ATc) at ($(ll)!.5!(ATur)$);
    \coordinate (cATc) at ($(ATlr)!.5!(Aul)$);
    \coordinate (UTc) at (cATc);
    \coordinate (WTdelim) at ($(ll)!.25!(ur)$);
    \coordinate (WTlr) at (ll -| WTdelim);
    \coordinate (WTur) at (WTdelim |- ur);
    \coordinate (WTc) at ($(ll)!.5!(WTur)$);
    \coordinate (cWTc) at ($(WTlr)!.5!(Aul)$);
    \coordinate (VTc) at (cWTc);
}
\newcommand{\pspace}[1][]{
    \draw[thick,#1] (ll) rectangle (ur);
}

\newcommand{\T}{\mathcal{T}}

\newcommand{\ulp}{\underline{p}}
\newcommand{\olp}{\overline{p}}
\newcommand{\ulh}{\underline{h}}
\newcommand{\olh}{\overline{h}}
\newcommand{\X}{\mathcal{X}}

\newcommand{\N}{\mathbb{N}}
\newcommand{\R}{\mathbb{R}}

\newcommand{\one}{\mathbbm{1}}
\newcommand{\PP}{\mathbb{P}}

\newcommand{\U}{\mathscr{U}}
\newcommand{\V}{\mathscr{V}}

\newcommand{\W}{\mathscr{W}}
\newcommand{\C}{\mathcal{C}}

\newcommand{\A}{\mathscr{A}}

\newcommand{\F}{\mathcal{F}}

\newcommand{\Pcal}{\mathcal{P}}
\newcommand{\EE}{\mathbb{E}}

\newcommand{\uppereach}{ \rightharpoonup }
\newcommand{\lowereach}{ \rightharpoondown 
  }
  \newcommand{\ch}{\operatorname{co}}
\newcommand{\notuppereach}{\not\rightharpoonup}
\newcommand{\notlowereach}{ \not\rightharpoondown
  }

\usepackage{newtxmath}      

\usepackage{url}

\begin{document}
	\mainmatter          
	\title{Relation between Hitting Times and Probabilities for Imprecise Markov Chains}
	\titlerunning{Relation between Hitting Times and Probabilities for IMCs}  
	\author{Marco Sangalli \and Erik Quaeghebeur \and Thomas Krak}
	\authorrunning{Marco Sangalli et al.} 
	\tocauthor{Marco Sangalli, Erik Quaeghebeur, Thomas Krak}
	\institute{Eindhoven University of Technology, Eindhoven,\\
		\email{m.sangalli@tue.nl}}
	
	\maketitle           
\begin{abstract}
    In this paper, we investigate the relation between hitting times and hitting probabilities in discrete-time imprecise Markov chains (IMCs). We define lower and upper hitting times and probabilities for IMCs whose set of transition matrices $\T$ is compact, convex, and has separately specified rows. Building on reachability-based partitions of the state space, we prove two key implications: (i) finiteness of the upper expected hitting time entails the lower hitting probability equals one, and (ii) finiteness of the lower expected hitting time entails the upper hitting probability equals one. We further show an equivalence: the upper expected hitting time is finite if and only if the lower hitting probability is one. Finally, by presenting a counterexample, we show that the converse of the second implication can fail.
		\keywords{Markov chain, Imprecise Markov chain, Imprecise probability, Hitting time, Hitting probability}
	\end{abstract}
    \section{Introduction}
    Hitting times and hitting probabilities are two complementary lenses through which one studies reachability in stochastic systems. The former captures how long it takes a process to reach a designated target set, while the latter captures the chance that the process ever does so. In the classical, precise theory of homogeneous Markov chains on a finite state space these two notions are tightly linked: finiteness of the expected hitting time is equivalent to the hitting probability being one~\cite{kemeny1976finite,levin2009markov}.

    When uncertainty is present and dynamics are described not by a single transition matrix but by a set of transition matrices $\T$, the relationship between hitting times and probabilities is substantially more subtle. Replacing (precise) expectations and probabilities with lower and upper expectations and probabilities naturally produces four distinct objects: lower and upper expected hitting times, and lower and upper hitting probabilities~\cite{krak2019hitting,krak2021comphit}. Even though some one-way implications survive, the neat equivalences from the precise case need not hold in the imprecise framework. Understanding which implications persist and which fail is both theoretically interesting and practically important for robust reachability analysis under model uncertainty.

    A fruitful approach to understanding these phenomena in discrete-time imprecise Markov chains (IMCs) is to partition the state space according to reachability properties relative to $\T$. Following the work of Sangalli et al.~\cite{sangalli2026computing,sangalli2025meeting}, one distinguishes states that never lower reach the target, $\A_\T$, states that never upper reach it, $\W_\T$, and further classes $\U_\T$ and $\V_\T$ that capture intermediate behaviours. These reachability-based classes allow a clean characterisation of when lower and upper hitting times are finite and when lower and upper hitting probabilities are strictly positive.

    The aim of this paper is to provide a summary of these recently established findings on hitting times and probabilities for imprecise Markov chains, to draw out the connection between them, and to complete this line of research with two novel results.
    Specifically, we show that (i) finiteness of the lower (resp. upper) hitting time implies that the upper (resp. lower) hitting probability equals one (Theorem \ref{implication1}), and (ii) the upper expected hitting time is finite if and only if the lower hitting probability is one (Theorem \ref{th:uhtandlhp}). 
    We further demonstrate, by means of a counterexample, that the equivalence in (ii) is lost when ``lower'' and ``upper'' are swapped.
   
    The remainder of the paper is structured as follows: Sections~\ref{sec:hmcs} and \ref{sec:expected} recall basic notation and classical results on hitting times and hitting probabilities for homogeneous Markov chains. Section~\ref{sec:IMCs} introduces imprecise Markov chains and lower and upper hitting times and probabilities. 
     Section~\ref{sec:relation} discusses the relation between expected hitting times and hitting probabilities in this imprecise setting. We end with concluding remarks in Section~\ref{sec:conclusion}.
    
    \section{Preliminaries}\label{sec:preli}
    In this section, we present all the preliminary concepts needed to study expected hitting times and hitting probabilities for IMCs~\cite{sangalli2026computing,krak2019hitting,norrisbook1997markov,kemeny1976finite}.
    \subsection{Stochastic Processes and Markov Chains}\label{sec:hmcs} 
    Let $\N$ denote the positive integers and set $\N_0 := \N\cup\{0\}$.
    Let $\X$ be a discrete finite set, and let $2^\X$ be its power set. Define the sample space $\Omega=\X^{\N_0}$ and let $\F\coloneqq (2^\X)^{\otimes\N_0}$ be the $\sigma$-algebra generated by cylinder events. 
    A discrete-time stochastic process taking values in $\X$ is a sequence of random variables $(X_n)_{n\in\N_0}$, $X_n : (\Omega, \F) \to (\X, 2^{\X})$ where $X_n(\omega)=\omega(n)$ for all $\omega \in \Omega$ and all $n\in \N_0$. Denote by $\PP_X$ the probability measure on $(\Omega, \F)$ representing the law of $(X_n)_{n\in\N_0}$.

    The process $(X_n)_{n\in\N_0}$ is called a \textit{Markov chain} if it satisfies 
\begin{equation}\label{eq:markovprop}
    \PP_X(X_{n+1}=x_{n+1}\mid X_{0:n}=x_{0:n})=\PP_X(X_{n+1}=x_{n+1}\mid X_{n}=x_{n})
\end{equation}
for all $x_0, \dots, x_n,x_{n+1} \in \X$ and all $n\in\N_0$. 
A Markov chain is (time-)\textit{homogeneous} if transition probabilities do not depend on $n$; that is, if
\begin{equation*}
   \PP_X(X_{n+1}=y\mid X_n=x)=\PP_X(X_1=y\mid X_0=x) 
\end{equation*}
for all $x,y\in \X$ and all $n\in \N_0$. 
The stochastic matrix $T\in \R^{|\X|\times |\X|}$ defined as $T(x,y)\coloneqq \PP_X(X_1=y\mid X_0=x)$ is the \textit{transition matrix} of the homogeneous Markov chain $(X_n)_{n\in\N_0}$ and uniquely determines the behaviour of the process up to its initial distribution. 
\subsection{Expected hitting times and hitting probabilities}\label{sec:expected}
Fix a nonempty target set $A\subset\X$. The \textit{hitting time} $\tau_A: (\Omega, \F, \PP_X) \to \N_0 \cup \{+ \infty\}$ of a homogeneous Markov chain $(X_n)_{n\in \N_0}$ is the random variable 
\begin{equation}
      \tau_A(\omega) :=\inf\{n\ge0: X_n(\omega)\in A\}.
\end{equation} 
Conditioned on the chain starting at $x\in\X$, the \emph{expected hitting time} is
\begin{equation}
  h^T(x):=\EE_{\PP_T}\bigl[\tau_A \mid X_0=x\bigr],
\end{equation}
where we write $\PP_T\coloneqq \PP_X$. 
Intuitively, $h^T(x)$ is the average number of steps needed to reach $A$ when the process starts from $x$. The vector of expected hitting times is the minimal nonnegative solution of a system of equations~\cite[{Theorem 1.3.5}]{norrisbook1997markov}:
\begin{equation}\label{eq:system_hittintimes}
\begin{cases}
    h^T(x)=0&\text{if $x\in A$,}\\
     h^T(x)=1+\textstyle\sum\limits_{y\in \X}T(x,y)h^T(y)&\text{if $x\notin A$.}
\end{cases}
\end{equation}
We define the \textit{hitting probability} of the homogeneous Markov chain $(X_n)_{n\in\N_0}$ as
\begin{equation}
    p^T(x):=\PP_T(\tau_A<+\infty\mid X_0=x).
\end{equation}
Hitting probabilities are the minimal nonnegative solution of~\cite[{Theorem 1.3.2}]{norrisbook1997markov}
\begin{equation}
    \begin{cases}
        p^T(x)=1&\text{if $x\in A$,}\\
        p^T(x)=\textstyle\sum\limits_{y\in \X}T(x,y)p^T(y)&\text{if $x\notin A$.}
    \end{cases}
\end{equation}
There exists a well-known relation between expected hitting times and hitting probabilities~\cite{kemeny1976finite,levin2009markov}: the expected hitting time is finite if and only if the hitting probability equals one. 
Since in the literature this result is often left implicit, we provide the whole argument below and we split it into its two implications for later reference.
\begin{proposition}\label{prop:rightimpl}
    Let $\X$ be a finite state space, let $(X_n)_{n\in \N_0}$ be a homogeneous Markov chain on $\X$ with transition matrix $T$, and let $x\in \X$. Then 
    \begin{equation}
        h^T(x)<+\infty \ \Rightarrow \ p^T(x)=1.
    \end{equation}
\end{proposition}
\begin{proof}
    Recall that $h^T(x)=\EE_{\PP_T}\bigl[\tau_A \mid X_0=x\bigr]<+\infty$. Since the nonnegative random variable $\tau_A$ has finite expected value, its probability of being infinite is zero and $p^T(x)=\PP_T(\tau_A<+\infty \mid X_0=x)=1$. \qed
\end{proof}
\begin{proposition}\label{prop:leftimpl}
    Let $\X$ be a finite state space, let $(X_n)_{n\in \N_0}$ be a homogeneous Markov chain on $\X$ with transition matrix $T$, and let $x\in \X$. Then 
    \begin{equation}
        h^T(x)<+\infty \ \Leftarrow \ p^T(x)=1.
    \end{equation}
\end{proposition}
\begin{proof}

Define $\C_T = \{y\in A^c \mid \forall n\in \N, [T^n\one_{A}](y)=0\}$. 
Since $p^T(x)=1$, 
we have $x\notin \C_T$. If $x\in A$, $h^T(x)=0$ and the claim follows.
If $x\in A^c\setminus \C_T$, the system in \eqref{eq:system_hittintimes} has a unique solution:
\[
h^T(x)= \sum_{y\in A^c \setminus \C_T}(I-T|_{A^c\setminus \C_\T})^{-1}(x,y),
\]
where $I$ denotes the identity matrix. The invertibility of $I-T|_{A^c\setminus \C_\T}$ is guaranteed by Sangalli et al.~\cite[{Lemma 2.12}]{sangalli2026computing}. It follows that $h^T(x)<+\infty$. \qed
\end{proof}

We seek to explore whether this, or a similar, relation holds true in the imprecise framework, when instead of a single (precise) process we are dealing with a set of homogeneous Markov chains.
\subsection{Imprecise Markov chains}\label{sec:IMCs}
Rather than a single homogeneous Markov chain $(X_n)_{n\in\N_0}$ governed by a transition matrix $T$, we now consider a family of these processes. Given a set $\T\subset \R^{|\X|\times |\X|}$, we consider the set $\Pcal$ made of all homogeneous Markov chains with transition matrix contained in the set $\T$. The set $\Pcal$ is said to be an \textit{imprecise Markov chain} (IMC)~\cite{DeCooman2009,HermansSkulj2014,krak2019hitting,sangalli2026computing}.  
This choice of $\Pcal$ is justified by Krak et al.~\cite[{Theorem 12 and 18}]{krak2019hitting} showing that lower and upper hitting times and hitting probabilities remain unchanged when the set $\Pcal$ is enlarged to include all (time-) inhomogeneous non-Markov processes that are locally compatible with $\T$. 
In the following, we assume that the set $\T$ is nonempty, compact, convex, and with separately specified rows (SSR)~\cite{HermansSkulj2014,krak2019hitting,sangalli2026computing}.

For an imprecise Markov chain $\Pcal$, we define lower and upper expectations:
\[
\underline{\EE}_\Pcal[\cdot\mid\cdot] \coloneqq \inf_{\PP_X\in \Pcal} \EE_{\PP_X}[\cdot\mid\cdot] \quad \text{and} \quad \overline{\EE}_\Pcal[\cdot\mid\cdot] \coloneqq \sup_{\PP_X\in \Pcal} \EE_{\PP_X}[\cdot\mid\cdot].
\]
These operators provide conservative, tight bounds for any quantity of interest.

For a fixed nonempty target set $A\subset \X$, we define the \textit{lower and upper expected hitting times} for the imprecise Markov chain parametrised by $\T$ as 
\begin{align*}
    &\ulh^\T(x)\coloneqq \underline{\EE}_\Pcal[\tau_A\mid X_0=x] = \textstyle\inf_{T\in \T} h^T(x);\\
    &{\olh^\T(x)\coloneqq \overline{\EE}_\Pcal[\tau_A\mid X_0=x] = \textstyle\sup_{T\in \T}h^T(x).}
\end{align*}
Analogously, we define \textit{lower and upper hitting probabilities} as
\begin{align*}
    &\ulp^\T(x)\coloneqq \underline{\EE}_\Pcal[\one_{\{\tau_A<+\infty\}}\mid X_0=x] = \textstyle\inf_{T\in \T}p^T(x);\\
    &{\olp^\T(x)\coloneqq \overline{\EE}_\Pcal[\one_{\{\tau_A<+\infty\}}\mid X_0=x] = \textstyle\sup_{T\in \T}p^T(x).}
\end{align*}
\section{Relation between Hitting Times and Hitting Probabilities for IMCs}\label{sec:relation}
We can partition the state space $\X$ into sets of states that have infinite lower or upper hitting time and states with lower or upper hitting probability equal to zero. To do so, we introduce the concept of reachability for IMCs~\cite{sangalli2026computing,debock2025convergent}: 
\begin{itemize}
    \item $x$ \textit{lower reaches} $A$, denoted by $x\lowereach A$, if for all $T\in \T$ there exists $\smash{n\in \N_0}$ such that $\smash{[T^n\one_A](x)>0}$;
    \item $x$ \textit{upper reaches} $A$, denoted by $x\uppereach A$, if there exist $T\in \T$ and $\smash{n\in \N_0}$ such that $\smash{[T^n\one_A](x)>0}$.
\end{itemize}
We can define the following sets of states: 
\begin{align*}
    \A_\T\coloneqq\{x\in A^c \mid x\notlowereach A \}, &&
    \U_\T\coloneqq\{x\in A^c\setminus \A_\T \mid  x\uppereach \A_\T\},\\
    {\W_{\T}} \coloneqq \{ x\in A^c \mid x\notuppereach A\}, &&
    \V_\T\coloneqq\{x\in A^c\setminus \W_\T \mid  x\lowereach \W_\T\}.
\end{align*}
\begin{figure}[h]
    \centering
    \begin{tikzpicture}[yscale=1.72] 
        \begin{scope}
            \initialize
            \pspace[fill=lightgray]
            \node at ($(ll)!.5!(ur)$) {\(\X\)};
        \end{scope}
        \begin{scope}[yshift=-1.7cm]
            \initialize
            \fill[lightgray] (All) rectangle (ur);
            \draw[thick] (All) -- (Aul);
            \node at (Ac) {\(A\)};
            \pspace
        \end{scope}
        \begin{scope}[xshift=4cm]
            \initialize
            \fill[lightgray] (ll) rectangle (ATur);
            \draw[thick] (All) -- (Aul);
            \draw[thick,dashed] (ATlr) -- (ATur);
            \node at (ATc) {\(\A_\T\)};
            \draw[->,>={Stealth[right]}] (cATc) -- (Ac);
            \pspace
        \end{scope}
        \begin{scope}[xshift=7.7cm]
            \initialize
            \fill[lightgray] (ATlr) |- (Aul) --cycle;
            \draw[thick] (All) -- (Aul);
            \draw[thick,dashed] (ATlr) -- (ATur);
            \draw[dashed] (ATlr) -- (Aul);
            \node at ([yshift=.5cm]UTc) {\(\U_\T\)};
            \draw[->,>={Stealth[right]}] ([xshift=-.2cm]UTc) -- (ATc);
            \pspace
        \end{scope}
        \begin{scope}[xshift=4cm,yshift=-1.7cm]
            \initialize
            \fill[lightgray] (ll) rectangle (WTur);
            \draw[thick] (All) -- (Aul);
            \draw[thick,dotted] (WTlr) -- (WTur);
            \node at (WTc) {\(\W_\T\)};
            \draw[->,>={Stealth[left]}] (cWTc) -- (Ac);
            \pspace
        \end{scope}
        \begin{scope}[xshift=7.7cm,yshift=-1.7cm]
            \initialize
            \fill[lightgray] (WTlr) |- (Aul) --cycle;
            \draw[thick] (All) -- (Aul);
            \draw[thick,dotted] (WTlr) -- (WTur);
            \draw[dotted] (WTlr) -- (Aul);
            \draw[thick,dashed,gray] (ATlr) -- (ATur);
            \draw[dashed,gray] (ATlr) -- (Aul);
            \node[above left=.1cm of VTc] {\(\V_\T\)};
            \draw[->,>={Stealth[left]}] ([xshift=-.4cm]VTc) -- (WTc);
            \pspace
        \end{scope}
    \end{tikzpicture}
    \caption{Illustration of key sets of states and their reachability characteristics. Left: the state space $\X$ (top) and the target set $A$ (bottom). Centre: $\A_{\T}$ (top) and $\W_\T$ (bottom), the sets of states that do not lower (resp. upper) reach the target. Right: $\U_{\T}$ (top) and $\V_\T$ (bottom), the sets of states that upper (resp. lower) reach $\A_\T$ (resp. $W_\T$). } \label{fig:key_sets_of_states}
\end{figure}

Sangalli et al.~\cite[Theorem 3]{sangalli2025meeting} showed that $\olh^\T(x)$ is infinite if and only if $x\in \A_\T \cup \U_\T$ and that $\ulh^\T(x)$ is infinite if and only if $x\in \W_\T \cup \V_\T$. This implies $\W_\T \cup \V_\T \subseteq \A_\T \cup \U_\T$. We also know from Sangalli et. al.~\cite[{Proposition 3.11 and 3.15}]{sangalli2026computing} that $\ulp^\T(x)=0$ if and only if $x\in \A_\T$ and that  $\olp^\T(x)=0$ if and only if $x\in \W_\T$. This implies $\W_\T\subseteq \A_\T$. 
The relations between the sets $\A_\T$, $\U_\T$, $\W_\T$ and $\V_\T$ are illustrated in Fig. \ref{fig:key_sets_of_states}.

From these relations, we also deduce 
\begin{align*}
    \smash{\olh^\T(x)<+\infty \ \Rightarrow \ \ulp^\T(x)>0, \quad\qquad 
    \ulh^\T(x)<+\infty \ \Rightarrow \  \olp^\T(x)>0 .}
\end{align*}
The following result strengthens these implications.
\begin{theorem}\label{implication1}
    Let $\T$ be a nonempty, compact, convex, with SSR set of transition matrices over $\X$, and let $x\in \X$. Then
    \begin{enumerate}[label=\arabic*), leftmargin=1cm]
    \item $\ \olh^\T(x)<+\infty \Rightarrow \ulp^\T(x)=1$;
    \item $\ \ulh^\T(x)<+\infty \Rightarrow \olp^\T(x)=1$.
\end{enumerate}
\end{theorem}
\begin{proof}
 If the upper hitting time is finite, then, for all $T\in \T$, we have that $h^T(x)<+\infty$. By Proposition \ref{prop:rightimpl}, we have $p^T(x)=1$ for all $T\in \T$, therefore $\ulp^{\T}(x)=\inf_{T\in \T} p^T(x)=1$.
    
    Similarly, if the lower hitting time is finite, there exists $T\in \T$ such that $h^T(x)<+\infty$. By Proposition \ref{prop:rightimpl}, we have $p^T(x)=1$, whence $\olp^\T(x)=1$. \qed
\end{proof}
We now invert implications 1) and 2) in Theorem \ref{implication1} and either prove their correctness or find a counterexample in which they fail. The following result states that the converse of implication 1) holds true.
\begin{theorem}
    \label{th:uhtandlhp}
    The upper expected hitting time $\olh^\T$ is finite if and only if the lower hitting probability $\ulp^\T$ is one, i.e.
    \begin{equation}
        {\olh^{\T}(x)<+\infty \ \Leftrightarrow \ \ulp^\T(x)=1}
    \end{equation}  
\end{theorem}
\begin{proof}
    We need to prove that $\smash{\ulp^\T(x)=1 \ \Rightarrow \ \olh^\T(x)<+\infty}$ since we already know the converse.
    By hypothesis, we have $p^T(x)=1$, for all $T\in T$. Therefore, by Proposition \ref{prop:leftimpl}, $h^T(x)<+\infty$ for all $T\in \T$. By Krak et al.~\cite[Theorem 12]{krak2019hitting}, there exists a transition matrix $T^*\in \T$ such that $h^{T^*}=\sup_{T\in \T} h^T=\olh{^\T}$. Thus, we conclude that ${\olh{^\T}(x)=h^{T^*}(x)<+\infty}$. \qed
\end{proof}
The converse of implication 2) fails: the counterexample below shows that, despite the upper hitting probability being one, the lower hitting time is infinite.
\begin{example}\label{ex:unosun}
{This example is a variation of Sangalli et al.~\cite[Example 5]{sangalli2026computing}.}  Let $\X=\{1,2,3\}$ be the state space, and let $\{2\}$ be the target set. Let the set of transition matrices on $\X$ be
    \begin{equation*}
        \T=\ch\left\{I, T_n=
        \begin{bmatrix}
            1-\frac{1}{n}-\frac{1}{n^2} & \frac{1}{n} & \frac{1}{n^2}\\
            0 & 1 & 0\\
            0 & 0 & 1\\
        \end{bmatrix}
        : n\ge 2
        \right\},
    \end{equation*}
    where by ``co'' we denote the convex hull. 
    The transition behaviour of $T_n$ is represented by the following graph:
    
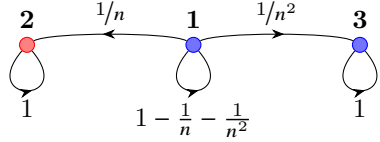
\begin{center}
    \begin{tikzpicture}[
        xscale=1.75,yscale=1.75,
        redNode/.style={circle, draw=red, fill=red!50, inner sep=2pt},
        blueNode/.style={circle,draw=blue, fill=blue!55,inner sep=2pt},
        every edge/.style={draw,looseness=0.3,
            postaction={
                decorate,
                decoration={
                    markings,
                    mark=at position 0.55 with {\arrow[#1]{Stealth}}
                }
            }
        }
    ]
        \node[blueNode,label=90:\textbf{1}] (1) {};
        \node[redNode,label=90:\textbf{2},left=2cm of 1] (2) {};
        \node[blueNode,label=90:\textbf{3},right=2cm of 1] (3) {};
        \path (1) edge[bend right=15] node[above, yshift=-0.5pt] {$\nicefrac{1}{n}$} (2);
        \path (1) edge[bend left=15] node[above, yshift=-0.5pt] {$\nicefrac{1}{n^2}$} (3);
        \path (1) edge[relative=false,out=230,in=-50,min distance=3.5ex] node[below, yshift=-0.1pt] {$1-\nicefrac{1}{n}-\nicefrac{1}{n^2}$} (1);
        \path (2) edge[relative=false,out=230,in=-50,min distance=3.5ex] node[below] {$1$} (2);
        \path (3) edge[relative=false,out=230,in=-50,min distance=3.5ex] node[below] {$1$} (3);
    \end{tikzpicture}
    \end{center}
    When starting from state ${1}$ and evolving under the transition matrix $T_n$, the hitting probability satisfies
    \begin{align*}
        p^{T_n}(1)=\left(1-\frac{1}{n}-\frac{1}{n^2}\right)p^{T_n}(1) +\frac{1}{n},
    \end{align*}
    Therefore, $p^{T_n}(1)=\nicefrac{n}{n+1} \to 1$ which is then the upper hitting probability. We observe that every transition matrix in $\T$ either disconnects states $1$ and $2$ or leaves a positive probability of never hitting the target. Therefore, the expected hitting time for every matrix in $\T$ is infinite, implying $\ulh^\T=+\infty$. \hfill$\diamond$
\end{example}
This counterexample arises because, as observed by Krak et al.~\cite{krak2019hitting}, there need not exist a transition matrix $T^*$ in $\T$ such that ${p^{T^*}=\olp^\T}$. In particular, in the present example
\[
{\lim_{n\to +\infty}p^{T_n}(1)=1\ne 0= p^I(1)=p \,^{\lim\limits_{n\to +\infty}T_n}(1),}
\]
and no other $T\in\T$ satisfies $p^{T}(1)=1$. 
\section{Conclusion}\label{sec:conclusion}
In this paper, we have studied the relation between lower and upper hitting times and probabilities for imprecise Markov chains.
In Theorem \ref{th:uhtandlhp}, we have proved an equivalence: the upper expected hitting time is finite if and only if the lower hitting probability is one. By contrast, Example \ref{ex:unosun} shows that this equivalence does not survive when ``lower'' and ``upper'' are interchanged.

We conclude the paper with Fig.~\ref{fig:relationship_overview}, which summarises the relationship between lower and upper hitting times and probabilities discussed and established in this work.
It uses the partitioning induced by the sets \(\A
_T\),  \(\U_\T\), \(\W_\T\), and \(\V_\T\), as illustrated in Fig.~\ref{fig:key_sets_of_states}.
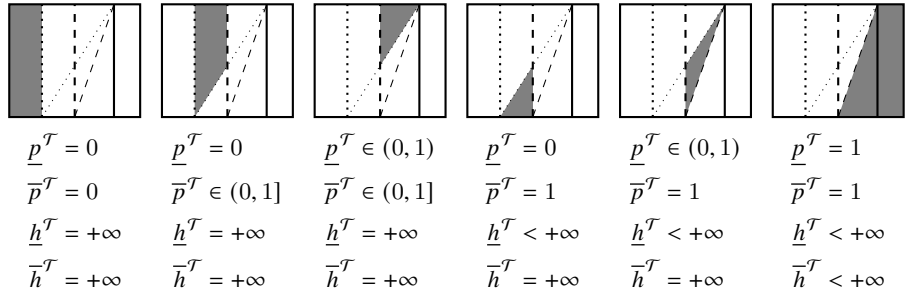
\begin{figure}[h]
    \centering
    \newcommand{\partitioning}{
        \draw[thick] (All) -- (Aul);
        \draw[thick,dashed] (ATlr) -- (ATur);
        \draw[dashed] (ATlr) -- (Aul);
        \draw[thick,dotted] (WTlr) -- (WTur);
        \draw[dotted] (WTlr) -- (Aul);
        \pspace        
    }
    \begin{tikzpicture}[xscale=.58]
        \begin{scope}
            \initialize
            \fill[gray] (ll) rectangle (WTur);
            \partitioning
            \node[below=0.1cm of basec] {
                \(\begin{aligned}
                    \ulp^\T &= 0\\
                    \olp^\T &= 0\\
                    \ulh^\T &= +\infty\\
                    \olh^\T &= +\infty
                \end{aligned}\)
            };
        \end{scope}
        \begin{scope}[xshift=3.5cm]
            \initialize
            \fill[gray] (WTlr) |- (Aul) -- cycle;
            \fill[white] (ATur) rectangle (All);
            \fill[white] (ATlr) -| (Aul) --cycle;
            \partitioning
            \node[below=0.1cm of basec] {
                \(\begin{aligned}
                    \ulp^\T &= 0\\
                    \olp^\T &\in (0,1]\\
                    \ulh^\T &= +\infty\\
                    \olh^\T &= +\infty
                \end{aligned}\)
            };
        \end{scope}
        \begin{scope}[xshift=10.5cm]
            \initialize
            \fill[gray] (WTur) rectangle (ATlr);
            \fill[white] (WTlr) |- (Aul) -- cycle;
            \partitioning
            \node[below=0.1cm of basec] {
                \(\begin{aligned}
                    \ulp^\T &= 0\\
                    \olp^\T &=1\\
                    \ulh^\T &< +\infty\\
                    \olh^\T &= +\infty
                \end{aligned}\)
            };
        \end{scope}
        \begin{scope}[xshift=7cm]
            \initialize
            \fill[gray] (WTlr) |- (Aul) -- cycle;
            \fill[white] (WTur) rectangle (ATlr);
            \partitioning
            \node[below=0.1cm of basec] {
                \(\begin{aligned}
                    \ulp^\T &\in(0,1) \\
                    \olp^\T &\in(0,1]\\
                    \ulh^\T &= +\infty\\
                    \olh^\T &= +\infty
                \end{aligned}\)
            };
        \end{scope}
        \begin{scope}[xshift=14cm]
            \initialize
            \fill[gray] (ATlr) |- (Aul) --cycle;
            \fill[white] (WTlr) |- (Aul) -- cycle;
            \partitioning
            \node[below=0.1cm of basec] {
                \(\begin{aligned}
                    \ulp^\T &\in (0,1)\\
                    \olp^\T &= 1\\
                    \ulh^\T &< +\infty\\
                    \olh^\T &= +\infty
                \end{aligned}\)
            };
        \end{scope}
        \begin{scope}[xshift=17.5cm]
            \initialize
            \fill[gray] (ATlr) -- (Aul) -- (ur) |-cycle;
            \partitioning
            \node[below=0.1cm of basec] {
                \(\begin{aligned}
                    \ulp^\T &= 1\\
                    \olp^\T &= 1\\
                    \ulh^\T &< +\infty\\
                    \olh^\T &< +\infty
                \end{aligned}\)
            };
        \end{scope}
    \end{tikzpicture}
    \caption{Overview of the relation between lower and upper hitting times and hitting probabilities.}\label{fig:relationship_overview}
\end{figure}

\subsubsection*{Acknowledgements}
This work has been partly supported by the PersOn project
(P21-03), which has received funding from Nederlandse Organisatie voor Wetenschappelijk Onderzoek (NWO). The authors thank the two anonymous reviewers for their careful reading and constructive suggestions, which helped improve the exposition.

\bibliographystyle{spbasic.bst}

\begin{thebibliography}{10}

\bibitem{DeCooman2009}
de Cooman, G., Hermans, F., Quaeghebeur, E.: Imprecise Markov chains and their limit behavior.
Probability in the Engineering and Informational Sciences 23(4), 597--635 (2009).
doi:10.1017/S0269964809990039

\bibitem{debock2025convergent}
De Bock, J., Erreygers, A., Persiau, F.: A convenient characterisation of convergent upper transition operators.
In: Proceedings of the 14th International Symposium on Imprecise Probabilities: Theories and Applications (ISIPTA).
Proceedings of Machine Learning Research, vol. 290, pp. 115--125 (2025).
doi:10.48550/arXiv.2502.04509

\bibitem{HermansSkulj2014}
Hermans, F., \v{S}kulj, D.: Stochastic processes.
In: Augustin, T., Coolen, F.P., de Cooman, G., Troffaes, M.C. (eds.) Introduction to Imprecise Probabilities,
chap. 11, pp. 258--278. Wiley (2014).
doi:10.1002/9781118763117.ch11

\bibitem{kemeny1976finite}
Kemeny, J.G., Snell, J.L.: {Finite Markov Chains: With a New Appendix ``Generalization of a Fundamental Matrix''}.
Undergraduate Texts in Mathematics. Springer-Verlag, New York, NY (1976)

\bibitem{krak2021comphit}
Krak, T.: Computing expected hitting times for imprecise Markov chains.
Space Technology Proceedings 8, 185--205 (2021).
doi:10.1007/978-3-030-80542-5\_12

\bibitem{krak2019hitting}
Krak, T., T'Joens, N., De Bock, J.: Hitting Times and Probabilities for Imprecise Markov Chains.
In: Proceedings of the 14th International Symposium on Imprecise Probabilities: Theories and Applications (ISIPTA).
Proceedings of Machine Learning Research, vol. 103, pp. 265--275 (2019).
1905.08781

\bibitem{levin2009markov}
Levin, D.A., Peres, Y., Wilmer, E.L.: {Markov Chains and Mixing Times}.
American Mathematical Society (2009).
doi:10.1090/mbk/107

\bibitem{norrisbook1997markov}
Norris, J.R.: {Markov Chains}.
Cambridge Series in Statistical and Probabilistic Mathematics. Cambridge University Press, Cambridge, UK (1997).
doi:10.1017/CBO9780511810633

\bibitem{sangalli2025meeting}
Sangalli, M., Quaeghebeur, E., Krak, T.: Upper Expected Meeting Times for Interdependent Stochastic Agents.
In: {Symbolic and Quantitative Approaches to Reasoning with Uncertainty}.
Lecture Notes in Artificial Intelligence, vol. 16099, pp. 238--252. Springer (2025).
doi:10.1007/978-3-032-05134-9\_17

\bibitem{sangalli2026computing}
Sangalli, M., Quaeghebeur, E., Krak, T.: Computing Lower and Upper Hitting Probabilities for Imprecise Markov Chains.
International Journal of Approximate Reasoning (2026).
doi:10.1016/j.ijar.2026.109671

\end{thebibliography}

\end{document}